\documentclass{article}
\usepackage{amssymb,amsmath,amsthm,graphicx,kotex} 

\def\qed{\hfill {\hbox{${\vcenter{\vbox{               
   \hrule height 0.4pt\hbox{\vrule width 0.4pt height 6pt
   \kern5pt\vrule width 0.4pt}\hrule height 0.4pt}}}$}}}

\def\tr{\triangleright}

\newtheorem{theorem}{Theorem}

\newtheorem{proposition}[theorem]{Proposition}

\theoremstyle{definition}
\newtheorem{example}{Example}
\newtheorem{definition}{Definition}
\newtheorem{remark}{Remark}

\date{}

\title{\Large \textbf{G-Family Polynomials}}

\author{
Madeline Brown\footnote{Email: mbrown2274@scrippscollege.edu.}
 \and
Sam Nelson\footnote{Email: Sam.Nelson@cmc.edu. Partially supported by
Simons Foundation Collaboration Grant 702597.}}

\begin{document}
\maketitle

\begin{abstract}
We introduce two notions of quandle polynomials for $G$-families of quandles: 
the quandle polynomial of the associated quandle and a $G$-family polynomial
with coefficients in the group ring of $G$. As an application we define
image subquandle polynomial enhancements of the $G$-family counting invariant
for trivalent spatial graphs and handlebody-links. We provide examples to
show that the new enhancements are proper.
\end{abstract}

\parbox{6in} {\textsc{Keywords:} Quandle polynomials, $G$-families of quandles,
enhancements of counting invariants, handlebody-links, trivalent spatial graphs

\smallskip

}

\section{\large\textbf{Introduction}}\label{I}

In \cite{J} and \cite{M} 
algebraic structures known as \textit{quandles} were introduced,
with axioms derivable from the Reidemeister moves in knot theory. To every
knot is associated a \textit{fundamental quandle}, also called the
\textit{knot quandle}, whose isomorphism class determines the knot complement 
up to (possibly orientation-reversing) homeomorphism, making it an extremely 
powerful knot invariant. Direct comparison of knot quandles given by 
presentations is difficult, but given any finite quandle $X$ the set of
quandle homomorphisms from the knot quandle to $X$ provides a number of easily 
computable knot and link invariants. In particular, cardinality of the homset
is an integer-valued invariant known as the \textit{quandle counting invariant}
$\Phi_X^{\mathbb{Z}}(K)$, and stronger invariants which determine the counting
invariant are known as \textit{enhancements}. For more detail
and further references, see \cite{EN}.

In \cite{I}, the notion of a \textit{$G$-family of quandles} for a group $G$
and set $X$ was introduced and applied to define invariants of 
handlebody-links and  trivalent spatial graphs. A $G$-family of quandles 
induces a quandle structure on the product $G\times X$ known as the 
\textit{associated quandle}, which raises an interesting question: given
a quandle $Q$, for which $G$-families is $Q$ the associated quandle?

In \cite{N}, the second listed author introduced the notion of 
\textit{quandle polynomials}, two-variable polynomial invariants of finite
quandles which reflect the distribution of trivial action 
throughout the quandle. These quandle invariants were then used to enhance 
the quandle counting invariants, providing a set of new knot and link 
invariants. Generalizations of the quandle polynomial were studied in 
\cite{CN,N2} and elsewhere.

In this paper we extend the quandle polynomial idea to the case of 
$G$-families of quandles. As an application we obtain new invariants of
handlebody-links and trivalent spatial graphs. The paper is organized
as follows. In Section \ref{QB} we reviews the basics of quandles and quandle
polynomials. In Section \ref{GF} we review the basics of $G$-families of
quandles and their handlebody-link and trivalent spatial graph invariants.
In Section \ref{GFM} we introduce our $G$-family polynomials and provide 
examples. As an application, in Section \ref{GFE} we introduce new enhanced
invariants of handlebody-links and trivalent spatial graphs using $G$-family 
polynomials. We end with some questions for future research in Section \ref{Q}.

\section{\large\textbf{Quandles and Quandle Polynomials}}\label{QB}

We begin with a definition; see \cite{J,M,EN} etc. for more.

\begin{definition}
Let $X$ be a set. A \textit{quandle operation} or \textit{quandle structure}
on $X$ is a binary operation $\tr:X\times X\to X$ such that 
\begin{itemize}
\item[(i)] For all $x\in X$, we have $x\tr x=x$,
\item[(ii)] For all $x,y\in X$ there is a unique $z\in X$ such that $x=z\tr y$,
and
\item[(iii)] For all $x,y,z\in X$, we have $(x\tr y)\tr z=(x\tr z)\tr(y\tr z)$.
\end{itemize}
These properties are known respectively as \textit{idempotence}, 
\textit{right-invertibility} and \textit{right self-distributivity}. Axiom (ii)
is equivalent to 
\begin{itemize}
\item[(ii$'$)] There is a binary operation $\tr^{-1}$ on $X$ such that for all $x,y\in X$ we have $(x\tr y)\tr^{-1}y=x$ and $(x\tr^{-1}y)\tr y=x$.
\end{itemize}
A set $X$ with a choice of quandle operation is a \textit{quandle}.
If the $\tr^{-1}$ operation is the same as the $\tr$ operation, i.e. if 
for all $x,y\in X$ we have $(x\tr y)\tr y=x$, we say $X$ is an 
\textit{involutory quandle} or \textit{kei}.
\end{definition}

\begin{example}
Standard examples of quandles include
\begin{itemize} 
\item The empty set $\emptyset$ is a quandle since the quandle
axioms are not existentially quantified,
\item Any singleton set $\{x\}$ is a quandle with $x\tr x=x$, 
\item More generally, any set $X$ is a quandle with the \textit{trivial 
quandle operation} $x\tr y=x$ for all $x,y\in X$,
\item Any group $G$ is a quandle with the \textit{core quandle operation} 
$x\tr y=yx^{-1}y$ for all $x,y\in X$,
\item Any group $G$ is a quandle with the \textit{$n$-fold conjugation 
quandle operation} $x\tr y=y^{-n}xy^n$  for all $x,y\in X$ for a choice of 
$n\in\mathbb{Z}$,
\item Any $\mathbb{Z}[t^{\pm 1}]$-module $M$ is a quandle with the 
\textit{Alexander quandle operation} $x\tr y=tx+(1-t)y$ for all $x,y\in M$,
\item Any vector space over a field of characteristic other than 2 
with symplectic form $[,]$ is a quandle with the
\textit{symplectic quandle operation} $x\tr y=x+[x,y]y$ for all $x,y\in X$.
\end{itemize}
We note that disjoint unions of quandles acting trivially on each other form
quandles and that every quandle can be decomposed as a disjoint union of 
\textit{orbit subquandles} acting on each other (not necessarily trivially).
\end{example}

\begin{definition}
Let $X$ and $Y$ be quandles. A map $f:X \to Y$ is a \textit{quandle 
homomorphism} if for all $x,x'\in X$ we have
\[f(x\tr x')=f(x)\tr f(x').\]
For finite quandles, a bijective quandle homomorphism is a \textit{quandle 
isomorphism}, and we say two quandles $X,Y$ are \textit{isomorphic} if there
exists an isomorphism $f:X\to Y$.
\end{definition}

Next, we review the notion of quandle polynomials. For more, see \cite{EN} or 
\cite{N}.

\begin{definition}
Let $X$ be a finite quandle. For each $x\in X$, define quantities $c(x)$ and
$r(x)$ by
\[\begin{array}{rcl}
c(x) & = & |\{y\in x\ | y\tr x=y\}| \\
r(x) & = & |\{y\in x\ | x\tr y=x\}|.
\end{array}\]
Then the \textit{quandle polynomial} of $X$ is the two-variable polynomial
\[\phi(X)=\sum_{x\in X} t^{c(x)}s^{r(x)}.\]
\end{definition}

\begin{example}
Let $X=\{1,2,3\}$ have the quandle structure given by the operation table
\[\begin{array}{r|rrr}
\tr & 1 & 2 & 3 \\ \hline
1 & 1 & 1 & 2 \\
2 & 2 & 2 & 1 \\
3 & 3 & 3 & 3.
\end{array}\]
To compute $\phi(X)$, for each element of $X$ we count the number of times
the row number appears in the row and column corresponding to the element.
In this case, looking in row $1$ we see the row number $1$ twice and
in column  $1$ we see the row number three times, so the element 
$1\in X$ contributes $t^3s^2$ to $\phi(X)$. Similarly, the element 
$2\in X$ contributes $t^3s^2$ and the element $3\in X$ contributes $t^1s^3$, 
so we have
\[\phi(X)=ts^3+2t^3s^2.\]
\end{example}

In \cite{N} we have the following result:
\begin{theorem}
If $X$ and $Y$ are isomorphic, then $\phi(X)=\phi(Y)$.
\end{theorem}

In \cite{N} it was shown by direct calculation that all quandles of order up to
five were distinguished from each other by their quandle polynomials, raising
the hope that the quandle polynomial might determine the quandle up to
isomorphism class; however, in \cite{CN}, examples were found of nonisomorphic
quandles of order six with the same quandle polynomial. Quandle polynomials
were generalized to rack polynomials in \cite{CN} and biquandle polynomials
in \cite{N2}. In the remainder of this paper, we will extend quandle polynomials
to the case of $G$-families of quandles.

\section{\large\textbf{G-Families of Quandles and Spatial Graphs}}\label{GF}

We begin this section with a definition from \cite{I}.

\begin{definition}
Let $G$ be a group and $X$ a set. A \textit{$G$-family of quandles} is a 
choice of quandle operation $\tr^g$ on $X$ for each element of $G$ such
that
\begin{itemize}
\item[(iv)] For all $x\in X$ and $g\in G$, $x\tr^g x=x$,
\item[(v)] For all $x,y\in X$ and $g,h\in G$, 
\[x\tr^{gh}y  = (x\tr^g y)\tr^h y \quad \mathrm{and}\quad x\tr^{1} y=x\]
where $1\in G$ is the identity, and
\item[(vi)] For all $x,y,z\in G$, 
\[(x\tr^g y)\tr^h z=(x\tr^h z)\tr^{h^{-1}gh} (y\tr^h z).\]
\end{itemize}
\end{definition}

\begin{example}
Any group $G$ can be regarded as a $G$-family
of singleton quandles $X=\{x\}$ with $x\tr^g x=x$ for all $g\in G$.
\end{example}

\begin{example}
Any group $G$ and set $X$ defines a $G$-family of quandles by setting
$x\tr^g y=x$ for all $x,y\in X$ and $g\in G$; this is the \textit{trivial
$G$-family} on $X$.
\end{example}

\begin{example}
A kei $X$ can be completed to a $\mathbb{Z}_2$-family of quandles where 
$\mathbb{Z}_2=\{1,t\ |\ t^2=1\}$ is the cyclic group of order 2 by
including a trivial quandle on the same set.
Let us define $x \tr^1y = x$ and $x\tr^t y=x\tr y$ where $\tr$ is the
quandle operation of $X$. Then we verify:
\begin{itemize}
\item[(iv)] For all $x\in X$, $x\tr^1x=x$ by definition of $\tr^1$ and 
$x\tr^tx=x$ since $X$ is a quandle,
\item[(v)] Let $x,y\in X$. There are four cases to verify:
\begin{eqnarray*}
x \tr^{1(1)} y & = & x \tr^1 y = (x\tr^1 y)\tr^1y, \\
x \tr^{t(1)} y & = & x \tr^t y = (x\tr^t y)\tr^1y, \\
x \tr^{1(t)} y & = & x \tr^y y = (x\tr^1 y)\tr^ty \ \mathrm{and} \\
x \tr^{t(t)} y & = & x \tr^1 y =x = (x\tr^t y)\tr^ty 
\end{eqnarray*}
where the last condition holds since $X$ is a kei.
\end{itemize}
\end{example}

Thus, the notion of $G$-families of quandles can be considered a generalization
of kei to larger groups $G$.

\begin{example}
We can specify a finite $G$-family of quandles with operation tables for the
group $G$ and the operations $\tr^g$. For example, the data
\[
\begin{array}{r|rrr}
\cdot & 1 & 2 & 3 \\ \hline
1 & 1 & 2 & 3 \\
2 & 2 & 3 & 1 \\
3 & 3 & 1 & 2
\end{array}\quad \begin{array}{r|rrrr} 
\tr^1 & 1 & 2 & 3 & 4 \\ \hline
1 & 1 & 1 & 1 & 1\\
2 & 2 & 2 & 2 & 2\\
3 & 3 & 3 & 3 & 3\\
4 & 4 & 4 & 4 & 4
\end{array}\quad \begin{array}{r|rrrr} 
\tr^2 & 1 & 2 & 3 & 4 \\ \hline
1 & 1 & 4 & 2 & 3\\
2 & 3 & 2 & 4 & 1\\
3 & 4 & 1 & 3 & 2\\
4 & 2 & 3 & 1 & 4
\end{array}\quad \begin{array}{r|rrrr} 
\tr^3 & 1 & 2 & 3 & 4 \\ \hline
1 & 1 & 3 & 4 & 2 \\
2 & 4 & 2 & 1 & 3 \\
3 & 2 & 4 & 3 & 1 \\
4 & 3 & 1 & 2 & 4
\end{array}
\]
defines a $\mathbb{Z}_3$-family of quandles.
\end{example}

As noted in \cite{I}, if $(G,X)$ is a $G$-family of quandles, then the
set $G\times X$ is a quandle under the operation
\[(g,x)\tr (h,y)=(h^{-1}gh,x\tr^h y).\]
This quandle structure is known as the \textit{associated quandle} of the
$G$-family $(G,X)$.

\begin{example}
Let $G$ be a group and $X=\{x\}$ a singleton quandle. Then there is an
isomorphism of quandles between the associated quandle of the $G$-family
$(G,X)$ and the conjugation quandle of $G$ given by $f(g,x)=g$:
\[f((g,x)\tr (h,x))=f(h^{-1}gh,x\tr^h x)=h^{-1}gh=g\tr h=f(g,x)\tr f(h,x).\]
\end{example}

We note that not every quandle is the associated quandle of a
$G$-family of quandles. For example, the quandle structure on the set
$X=\{1,2,3\}$ given by the operation table
\[\begin{array}{r|rrr}
\tr & 1 & 2 & 3 \\ \hline
1 & 1 & 1 & 2\\
2 & 2 & 2 & 1\\
3 & 3 & 3 & 3
\end{array}\]
is a quandle of cardinality 3, so we would need either $G=\{1\}$ and
$|X|=3$ or $G=\mathbb{Z}_3$ and $X=\{x\}$ a singleton
quandle. In the former case, the requirement of axiom (v) that $x\tr^1 y=x$
for all $x,y\in X$ is contradicted and in the latter case, the condition
\[(g,x)\tr (h,y)=(h^{-1}gh,x\tr^h y)\]
says
\[(g,x)\tr (h,x)=(-h+g+h,x)=(g,x)\]
and the associated quandle is trivial.

\begin{example}
The $\mathbb{Z}_2$-family of quandles given by 
\[
\begin{array}{r|rr}
\cdot & 1 & 2 \\\hline
1 & 1 & 2\\
2 & 2 & 1
\end{array}\ \
\begin{array}{r|rrr}
\tr_1 & 1 & 2 & 3 \\ \hline
1 & 1 & 1 & 1\\
2 & 2 & 2 & 2 \\
3 & 3 & 3 & 3
\end{array}\ \
\begin{array}{r|rrr}
\tr_2 & 1 & 2 & 3 \\ \hline
1 & 1 & 3 & 2\\
2 & 3 & 2 & 1\\
3 & 2 & 1 & 3
\end{array}
\]
has associated quandle
\[\begin{array}{r|llllll}
\tr & 1 & 2 & 3 & 4 & 5 & 6 \\ \hline
1 & 1 & 1 & 1 & 1 & 3 & 2 \\
2 & 2 & 2 & 2 & 3 & 2 & 1 \\
3 & 3 & 3 & 3 & 2 & 1 & 3 \\
4 & 4 & 4 & 4 & 4 & 6 & 5 \\
5 & 5 & 5 & 4 & 6 & 5 & 4 \\
6 & 6 & 6 & 6 & 5 & 4 & 6.
\end{array}\]
\end{example}

\begin{definition}
Let $(G,X)$ be a $G$-family of quandles and $Y\subset X$ a subset. 
We say $(G,Y)$ is a \text{$G$-subfamily of quandles} 
of $(G,X)$ if $(G,Y)$ is a $G$-family of quandles under the operations 
$\tr^g$ inherited from $(G,X)$.
\end{definition}

\begin{example}
In the $\mathbb{Z}_2$-family given by the operation tables
\[
\begin{array}{r|rr}
\cdot & 1 & 2  \\ \hline
1 & 1 & 2  \\
2 & 2 & 1  \\
\end{array}\quad \begin{array}{r|rrr} 
\tr^1 & 1 & 2 & 3  \\ \hline
1 & 1 & 1 & 1 \\
2 & 2 & 2 & 2 \\
3 & 3 & 3 & 3 \\
\end{array}\quad \begin{array}{r|rrr} 
\tr^2 & 1 & 2 & 3  \\ \hline
1 & 1 & 1 & 2\\
2 & 2 & 2 & 1\\
3 & 2 & 3 & 3 \\
\end{array}\ 
\]
we have $G$-subfamilies with $X=\{1,2\}$, $X=\{1\}$, 
$X=\{2\}$ and $X=\{3\}$.
\end{example}

$G$-families of quandles were used in \cite{I}
for distinguishing \textit{$Y$-oriented trivalent spatial graphs} 
and their quotient objects, \textit{handlebody-links}.
More precisely, let $K$ be a diagram consisting of oriented classical
crossings and trivalent vertices where we disallow sources and sinks.

Two trivalent spatial graph diagrams represent ambient isotopic spatial graphs
if they are related by a sequence of the following moves:
\[\includegraphics{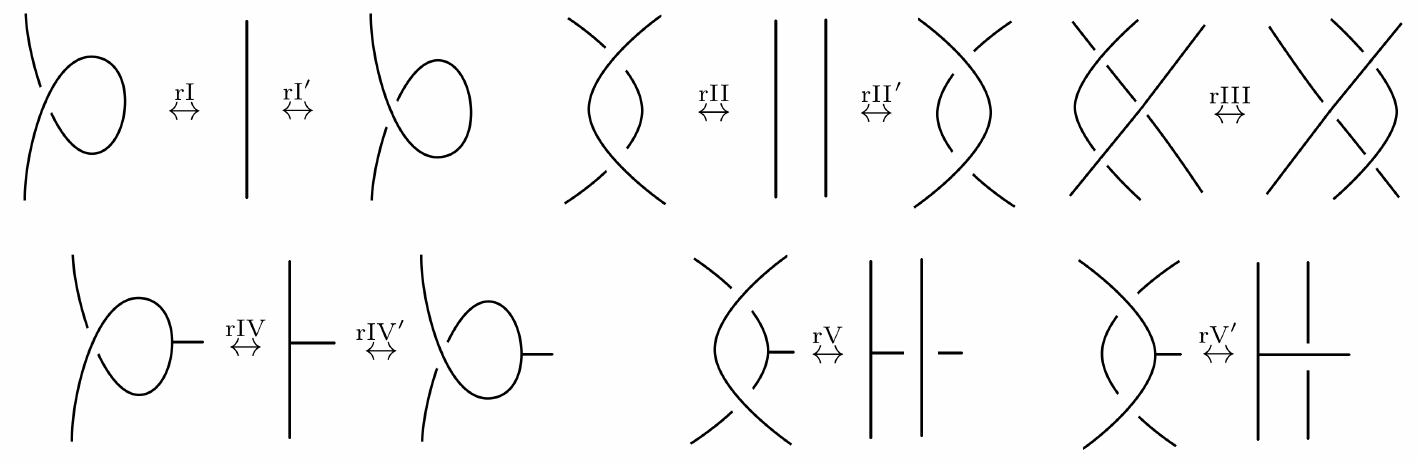}\]

Additionally allowing the \textit{IH-move}
\[\includegraphics{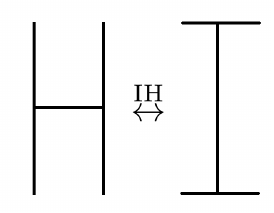}\]
give us \textit{handlebody-links} as a quotient of trivalent spatial graphs. 
See \cite{I} for more.

Then given a $G$-family of quandles $(G,X)$, a \textit{$(G,X)$-coloring} of
$K$ is an assignment of a pair $(g,x)$ to each arc in $K$ such that at crossings
and vertices, the following conditions are satisfied:
\[\includegraphics{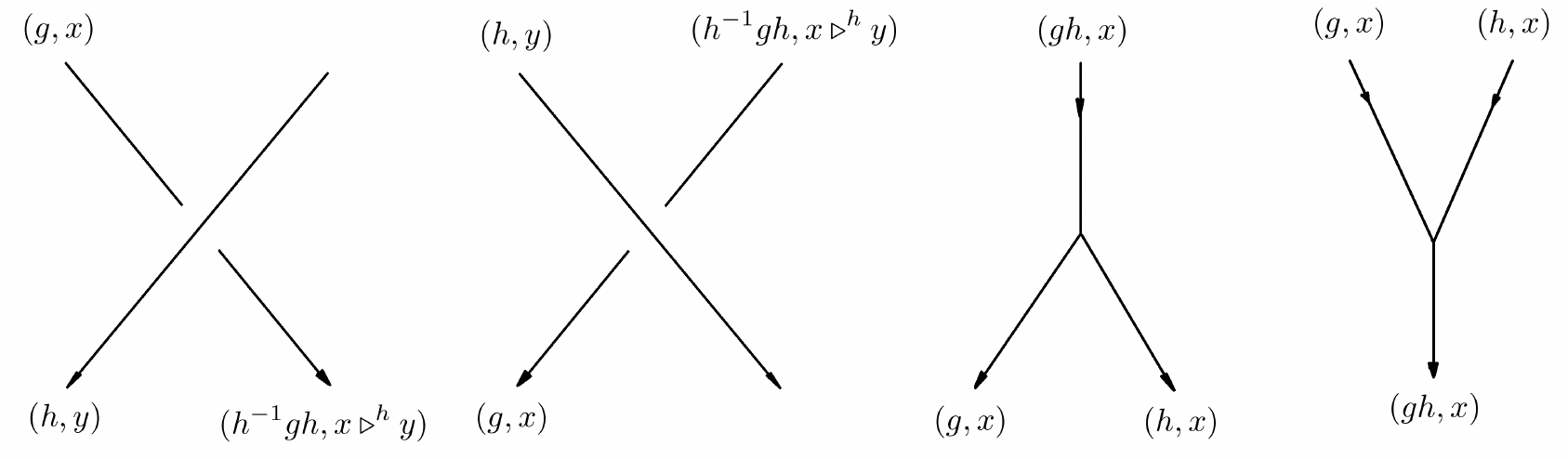}\]

Then as shown in \cite{I}, each $(G,X)$-coloring of $K$ before a
Reidemeister move corresponds to a unique $(G,X)$-coloring of 
$K$ after the move, and we have:

\begin{theorem}
Let $(G,X)$ be a $G$-family of quandles. The cardinality of the set
$\mathcal{C}((G,X),K)$ of $(G,X)$-colorings of a $Y$-oriented trivalent
spatial graph $K$ is an invariant of $Y$-oriented trivalent spatial graphs 
and $Y$-oriented handlebody-knots 
under ambient isotopy. This is known as the \textit{$G$-family counting 
invariant}, denoted 
\[\Phi_{(G,X)}^{\mathbb{Z}}(K)=|\mathcal{C}((G,X),K)|.\]
\end{theorem}

\begin{example}\label{ex:theta}
Consider the \textit{theta graph}, one of two unknotted trivalent spatial
graphs with two vertices, and let $(G,X)$ be the $\mathbb{Z}_2$-family given by
\[
\begin{array}{r|rr}
G & 1 & 2  \\ \hline
1 & 1 & 2  \\
2 & 2 & 1  \\
\end{array}\quad \begin{array}{r|rrr} 
\tr^1 & 1 & 2 & 3  \\ \hline
1 & 1 & 1 & 1 \\
2 & 2 & 2 & 2 \\
3 & 3 & 3 & 3 \\
\end{array}\quad \begin{array}{r|rrr} 
\tr^2 & 1 & 2 & 3  \\ \hline
1 & 1 & 3 & 2\\
2 & 3 & 2 & 1\\
3 & 2 & 1 & 3 \\
\end{array}.
\]
Then there are 12 $X$-colorings of this graph as shown:
\[\scalebox{1.5}{\includegraphics{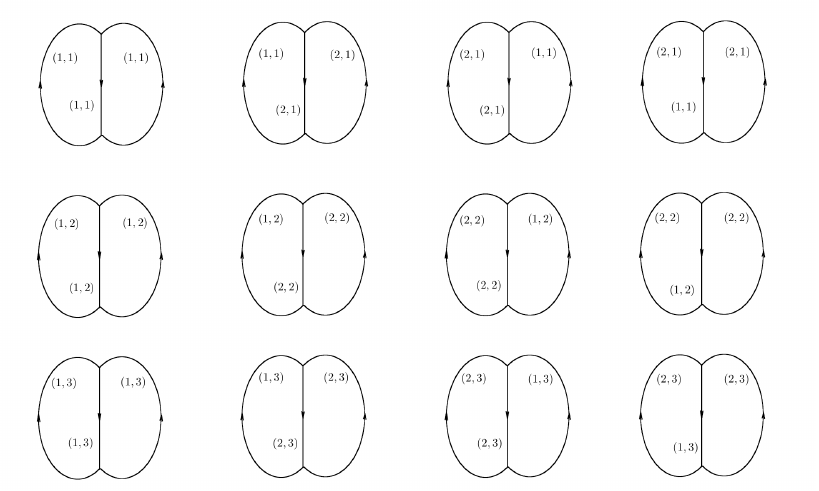}}\]
Thus, the G-family counting invariant for this theta graph with respect
to $X$ is $\Phi_{(G,X)}^{\mathbb{Z}}(\theta)=12.$
\end{example}

Recall that the \textit{image subquandle} of a quandle homomorphism
$f:X\to Y$ is the set of elements $y=f(x)\in Y$ which are the images of
elements $x\in X$ under $f$.

\begin{definition}
The $G$-subfamily generated by the set of pairs $(g,x)\in (G,X)$ appearing
in a particular $(G,X)$-coloring $f\in\mathcal{C}((G,X),K)$, i.e., the 
smallest $G$-subfamily of $(G,X)$ containing all the pairs $(g,x)$ in the 
coloring, is the \textit{image $G$-subfamily} of the coloring, denoted 
$\mathrm{Im}(f)$.
\end{definition}

\begin{example}
The image $G$-subfamilies for the colorings in Example \ref{ex:theta}
are $\{(1,1),(2,1)\}$ for the four colorings in the top row, 
$\{(1,2),(2,2)\}$ for the four colorings in the middle row and
$\{(1,3),(2,3)\}$ for the four colorings in the last row.
\end{example}

\section{\large\textbf{G-Family Polynomials}}\label{GFM}

Let us now turn to the question of quandle polynomials for $G$-families of
quandles. The simplest approach is to use the \textit{associated quandle}.
As observed in \cite{I}, given a $G$-family of quandles, there is a
quandle structure on $G\times X$ given by
\[(g,x)\tr(g',x')=(g'^{-1}gg',x\tr^g x')\]
known as the \textit{associated quandle} of the $G$-family. We will denote
this quandle structure by $A(G,X)$. Since $A(G,X)$
is a quandle, it has a quandle polynomial. Thus, we have:

\begin{definition}
Let $(G,X)$ be a $G$-family of quandles. The \textit{associated quandle
polynomial} of $(G,X)$, denoted $\phi_{A(G,X)}$, is the quandle
polynomial of $A(G,X)$, i.e.,
\[\sum_{(g,x)\in A(G,X)}t^{|c(g,x)|}s^{|r(g,x)|}.\]
\end{definition}

\begin{example}
The $\mathbb{Z}_2$-family of quandles given by
\[
\begin{array}{r|rr}
\cdot & 1 & 2 \\ \hline
1 & 1 & 2 \\
2 & 2 & 1
\end{array}\quad \begin{array}{r|rrrr}
\tr_2 & 1 & 2 &3  &4 \\ \hline
1 & 1& 1& 1& 1 \\
2 & 2& 2& 2& 2 \\
3 & 3& 3& 3& 3 \\
4 & 4& 4& 4& 4
\end{array}\quad \begin{array}{r|rrrr}
\tr_2 & 1 & 2 &3  &4 \\ \hline
1 & 1& 3& 1& 1 \\
2 & 2& 2& 2& 2 \\
3 & 3& 1& 3& 3 \\
4 & 4& 4& 4& 4
\end{array}
\]
has associated quandle
\[
\begin{array}{r|rrrrrrrr}
\tr & 1 & 2 & 3 & 4 & 5 &6 & 7 & 8 \\ \hline
1 & 1& 1& 1& 1& 1& 3& 1& 1 \\
2 & 2& 2& 2& 2& 2& 2& 2& 2 \\
3 & 3& 3& 3& 3& 3& 1& 3& 3 \\
4 & 4& 4& 4& 4& 4& 4& 4& 4 \\
5 & 5& 5& 5& 5& 5& 7& 5& 5 \\
6 & 6& 6& 6& 6& 6& 6& 6& 6 \\
7 & 7& 7& 7& 7& 7& 5& 7& 7 \\
8 & 8& 8& 8& 8& 8& 8& 8& 8
\end{array}
\]
and hence associated quandle polynomial $\phi_{A(G,X)}=3t^8s^8+4t^8s^7+t^4s^8.$ 
\end{example}

A second approach to quandle polynomials for $G$-families makes use of the 
$G$-family structure to get a two-variable polynomial with coefficients in the
group ring $\mathbb{Z}[G]$ over $G$.

\begin{definition}
Let $(G,X)$ be a $G$-family of quandles. For each $g\in G$, we have a
quandle structure on $X$ given by $\tr^g$; let 
\[\begin{array}{rcl}
c_g(x) & = & |\{y\in x\ | y\tr^g x=y\}| \\
r_g(x) & = & |\{y\in x\ | x\tr^g y=x\}|.
\end{array}\]
The \textit{$G$-family polynomial}, denoted $\phi_{(G,X)}$, is given by
\[\sum_{(g,x)\in (G,X)} gt^{|c_g(x)|}s^{|r_g(x)|}.\]
\end{definition}

\begin{remark}
For $G$-families of quandles presented via operation tables, we will
specify elements of $\mathbb{Z}[G,t,s]$ as $|G|$-tuples of polynomials
in $\mathbb{Z}[s,t]$. For example, if  $G=\{g_1,g_2,g_3\}$ then the
polynomial 
\[\phi=3g_1t^2s+6g_2ts^2+2g_2ts+4g_3\]
will be expressed as
\[{}[3t^2s, 6ts^2+2ts, 4].\]
\end{remark}

\begin{example}
Consider the $\mathbb{Z}_3$-family of quandles with $X=\{1,2,3,4\}$ given by
\[
\begin{array}{r|rrr}
\cdot & 1 & 2 & 3 \\ \hline
1 & 1 & 2& 3 \\
2 & 2 & 3& 1 \\
3 & 3 & 1& 2 \\
\end{array}\ \ \begin{array}{r|rrrr}
\tr^1 & 1 & 2 & 3 & 4 \\ \hline
1 & 1& 1& 1& 1 \\
2 & 2& 2& 2& 2 \\
3 & 3& 3& 3& 3 \\
4 & 4& 4& 4& 4  
\end{array}\ \ \begin{array}{r|rrrr}
\tr^2 & 1 & 2 & 3 & 4 \\ \hline
1 & 1& 3& 4& 2 \\
2 & 4& 2& 1& 3 \\
3 & 2& 4& 3& 1 \\
4 & 3& 1& 2& 4  
\end{array}\ \ \begin{array}{r|rrrr}
\tr^3 & 1 & 2 & 3 & 4 \\ \hline
1 & 1& 4& 2& 3 \\
2 & 3& 2& 4& 1 \\
3 & 4& 1& 3& 2 \\
4 & 2& 3& 1& 4
\end{array}\]

The element $1\in X$ contributes $g_1t^4s^4+g_2ts+g_3ts$ to the $G$-family
polynomial; repeating for the other elements of $X$, we have
\[\phi=[4t^4s^4,4ts,s4t].\]

\end{example}

\begin{definition}
Let $(G,X)$ be a $G$-family of quandles and $(G,Y)$ a $G$-subfamily. Then:
\begin{itemize}
\item The \textit{associated subquandle polynomial} of $(G,Y)$ is the
subquandle polynomial $\phi_{A(G,Y)\subset A(G,X)}$ of the associated quandle
of $(G,Y)$ considered as a subquandle of the associated quandle of $(G,X)$, and
\item The \textit{$G$-subfamily polynomial} of $(G,Y)$, denoted
$\phi_{(G,Y)\subset(G,X)}$, is the sum of contributions of elements of $(G,Y)$ to 
$\phi_{(G,X)}$, i.e.
\[\phi_{(G,Y)\subset(G,X)}=\sum_{(g,x)\in (G,Y)} gt^{|c_g(x)|}s^{|r_g(x)|}.\]
\end{itemize}
\end{definition}

\section{\large\textbf{$G$-family Polynomial Enhancements}}\label{GFE}

We can now apply our definitions to enhance the $G$-family counting invariant.

\begin{definition}
Let $(G,X)$ be a $G$-family of finite quandles and let $K$ be a $Y$-oriented
trivalent spatial graph. Then we define the
\begin{itemize}
\item \textit{Associated subquandle polynomial invariant} of $K$ with respect 
to $(G,X)$ to be the multiset of associated subquandle polynomials of the
image subquandles of each coloring, i.e.
\[\Phi_{(G,X)}^{A}(K)=\{\phi_{A(G,\mathrm{Im}(f))\subset A(G,X)}\ |\ f\in\mathcal{C}((G,X),K)\}\]
and the
\item \textit{$G$-family subquandle polynomial invariant} of $K$ with respect 
to $(G,X)$ to be the multiset of $G$-family subquandle polynomials of the
image subquandles of each coloring, i.e.
\[\Phi_{(G,X)}(K)=\{\phi_{(G,\mathrm{Im}(f))\subset (G,X)}\ |\ f\in\mathcal{C}((G,X),K)\}.\]
\end{itemize}
\end{definition}

We then state our main result:

\begin{proposition}
Let $(G,X)$ be a $G$-family of quandles.
If two handlebody-link diagrams $L$ and $L'$ are related by Reidemeister moves,
then $\Phi_{(G,X)}(L)=\Phi_{(G,X)}(L')$ and $\Phi_{(G,X)}^A(L)=\Phi_{(G,X)}^A(L')$.
\end{proposition}

\begin{proof}
It suffices to observe that the Reidemeister moves do not change the image
$G$-subfamily and the image subquandle of the associated quandle for a
$(G,X)$-coloring of a handlebody-link diagram.
\end{proof}

\begin{example}\label{ex:41}
To illustrate the invariants, let us consider the trivalent
spatial graph below representing handlebody-knot $4_1$
with the $\mathbb{Z}_2$-family of quandles $X$ given by
\[
\begin{array}{r|rr}
\cdot & 1 & 2  \\ \hline
1 & 1 & 2  \\
2 & 2 & 1  \\
\end{array}\quad \begin{array}{r|rrr} 
\tr^1 & 1 & 2 & 3  \\ \hline
1 & 1 & 1 & 1 \\
2 & 2 & 2 & 2 \\
3 & 3 & 3 & 3 \\
\end{array}\quad \begin{array}{r|rrr} 
\tr^2 & 1 & 2 & 3  \\ \hline
1 & 1 & 3 & 2\\
2 & 3 & 2 & 1\\
3 & 2 & 1 & 3. \\
\end{array}
\]
There are eighteen $(G,X)$-colorings, including for example the two
depicted here:
\[\includegraphics{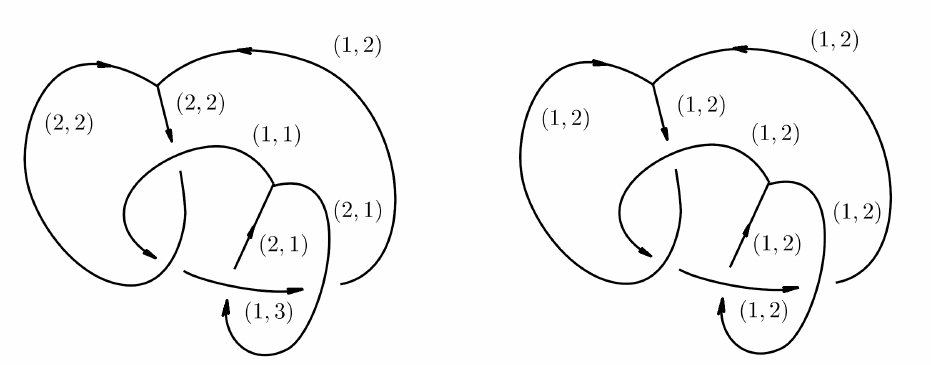}\]
The coloring on the right is a trivial coloring, since all arcs have 
the same color; the coloring on the left is nontrivial.

The image subfamily of the coloring on the left is the entire $G$-family,
while the image subfamily on the right is the subfamily $\{(1,2), (2,2)\}$.
Thus the left coloring contributes $3g_1t^3s^3+3g_2ts$ to the 
$G$-family subquandle polynomial invariant and the coloring on the right 
contributes $g_1t^3s^3+g_2ts$. Repeating for the other colorings, we obtain
invariant value
\[\Phi_{(G,X)}^G(4_1)=\{12\times[t^3s^3,ts],6\times[3t^3s^3,3ts]\}.\]
\end{example}

\begin{example}
Repeating the computation in Example \ref{ex:41} with the associated 
quandle 
\[\begin{array}{r|rrrrrr}
\tr & 1 & 2 & 3 & 4 & 5 & 6 \\ \hline
  1 & 1 & 1 & 1 & 1 & 3 & 2 \\
  2 & 2 & 2 & 2 & 3 & 2 & 1 \\
  3 & 3 & 3 & 3 & 2 & 1 & 3 \\
  4 & 4 & 4 & 4 & 4 & 6 & 5 \\
  5 & 5 & 5 & 5 & 6 & 5 & 4 \\
  6 & 6 & 6 & 6 & 5 & 4 & 6
\end{array}\]
of the $G$-family we obtain invariant value
\[\Phi_{(G,X)}^A(4_1)=\{3\times t^6s^4, 9\times (t^6s^4+t^2s^4), 6\times (3t^6s^4+3t^2s^4)\}.\]
\end{example}

\begin{example}
Using our \texttt{python} code, we compute that
the two handlebody-links 
\[\includegraphics{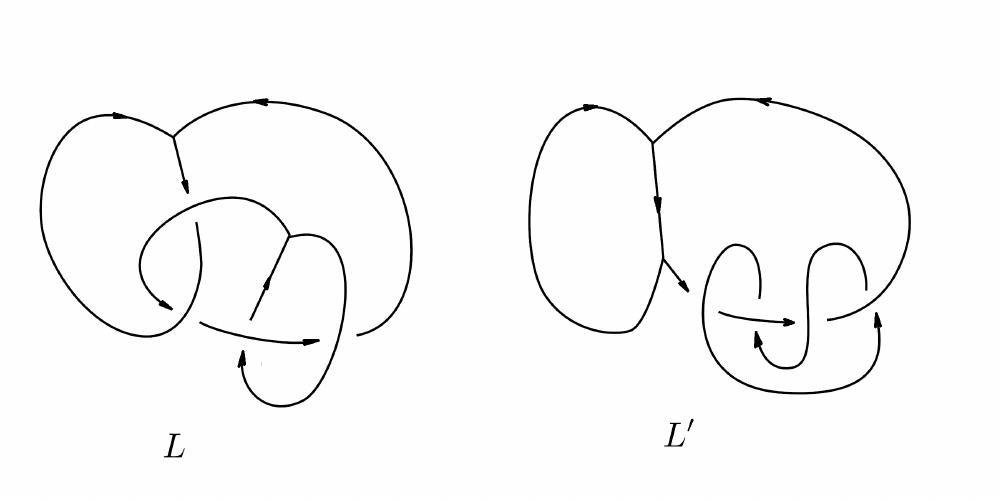}\]
both have 324 $(G,X)$-colorings by the 
$S_3$-family of quandles given by the operation tables
\[
\begin{array}{r|rrrrrr}
\cdot & 1 & 2 & 3 & 4 & 5 & 6 \\ \hline
1 & 1 & 2 & 3 & 4 & 5 & 6 \\
2 & 2 & 1 & 5 & 6 & 3 & 4 \\
3 & 3 & 6 & 1 & 5 & 4 & 2 \\
4 & 4 & 5 & 6 & 1 & 2 & 3 \\
5 & 5 & 4 & 2 & 3 & 6 & 1 \\
6 & 6 & 3 & 4 & 2 & 1 & 5 
\end{array}\] 
\[ \begin{array}{r|rrr}\tr^1 & 1 & 2 & 3 \\ \hline
1 & 1 & 1 & 1 \\ 2 & 2 & 2 & 2 \\ 3 & 3 & 3 & 3
\end{array} \ \begin{array}{r|rrr}\tr^2 & 1 & 2 & 3 \\ \hline
1 & 1 & 3 & 2 \\ 2 & 3 & 2 & 1 \\ 3 & 2 & 1 & 3
\end{array} \ \begin{array}{r|rrr}\tr^3 & 1 & 2 & 3 \\ \hline
1 & 1 & 3 & 2 \\ 2 & 3 & 2 & 1 \\ 3 & 2 & 1 & 3
\end{array} \ \begin{array}{r|rrr}\tr^4 & 1 & 2 & 3 \\ \hline
1 & 1 & 3 & 2 \\ 2 & 3 & 2 & 1 \\ 3 & 2 & 1 & 3
\end{array} \ \begin{array}{r|rrr}\tr^5 & 1 & 2 & 3 \\ \hline
1 & 1 & 1 & 1 \\ 2 & 2 & 2 & 2 \\ 3 & 3 & 3 & 3
\end{array} \ \begin{array}{r|rrr}\tr^6 & 1 & 2 & 3 \\ \hline
1 & 1 & 1 & 1 \\ 2 & 2 & 2 & 2 \\ 3 & 3 & 3 & 3
\end{array}\]

but are distinguished by their $\Phi_{(G,X)}$-values
\begin{eqnarray*}
\Phi_{(X,G)}(L) & = & \{162\times[t^3t^3,ts,ts,ts,t^3t^3,t^3t^3], 162\times[3t^3s^3,3ts,3ts,3ts,3t^3s^3,3t^3s^3] \}\\
\Phi_{(X,G)}(L') & = & \{108\times[t^3s^3,ts,ts,ts,t^3s^3,t^3s^3], 216\times[3t^3s^3,3ts,3ts,3ts,3t^3s^3,3t^3s^3]\}
\end{eqnarray*}
as well as by their $\Phi_{(G,X)}^A$-values
\begin{eqnarray*}
\Phi_{(X,G)}^A(L) & = & \{3\times t^{18}s^{12}, 27\times (t^{18}s^{12} + t^2s^4), 18\times (t^{18}s^{12} + t^9s^9), 18\times (t^{18}s^{12} + 3t^2s^4),\\
& &  36\times (t^{18}s^{12} + 2t^9s^9 + 3t^2s^4), 54\times (3t^{18}s^{12} + 3t^2s^4), 54\times (2t^9s^9 + 3t^2s^4),\\ & & 108\times(6t^9s^9 + 9t^2s^4), 6\times(t^{18}s^{12} + 2t^9s^9)\}
\\
\Phi_{(X,G)}^A(L') & = & \{
3\times t^{18}s^{12}, 6\times 2t^{18}s^{12}, 27\times(t^{18}s^{12} + t^2s^4), 30\times(t^{18}s^{12} + t^9s^9), \\ & &  18\times(2t^{18}s^{12} + t^2s^4),
18\times(3t^{18}s^{12} + t^2s^4), 18\times(t^{18}s^{12} + 3t^2s^4), \\ & & 
 72\times(t^{18}s^{12} + 2t^9s^9 + 3t^2s^4), 54\times(3t^{18}s^{12} + 3t^2s^4), \\ & &  12\times(2t^{18}s^{12} + t^9s^9), 30\times(t^{18}s^{12} + 2t^9s^9), \\ & &  36\times(2t^{18}s^{12} + 2t^9s^9 + 3t^2s^4)
\}
\end{eqnarray*}

In particular, this example shows that both $\Phi_{(X,G)}$ and $\Phi_{(X,G)}^A$ 
are not determined by the the number of $(G,X)$-colorings and hence are proper
enhancements.
\end{example}

\begin{example}
We selected an $S_3$-family of quandles with group and quandle tables
given by
\[
\begin{array}{r|rrrrrr} 
\cdot & 1 & 2 & 3 & 4 & 5 & 6 \\ \hline
1 & 1 & 2 & 3 & 4 & 5 & 6\\
2 & 2 & 1 & 5 & 6 & 3 & 4\\
3 & 3 & 6 & 1 & 5 & 4 & 2\\
4 & 4 & 5 & 6 & 1 & 2 & 3\\
5 & 5 & 4 & 2 & 3 & 6 & 1\\
6 & 6 & 3 & 4 & 2 & 1 & 5
\end{array}\ \ \begin{array}{r|rrrr}
\tr^1 & 1 & 2 & 3 & 4 \\ \hline
1 & 1 & 1 & 1 & 1 \\
2 & 2 & 2 & 2 & 2 \\
3 & 3 & 3 & 3 & 3 \\
4 & 4 & 4 & 4 & 4
\end{array} \ \ \begin{array}{r|rrrr}
\tr^2 & 1 & 2 & 3 & 4 \\\hline
1 & 1 & 1 & 2 & 2 \\
2 & 2 & 2 & 1 & 1 \\
3 & 4 & 4 & 3 & 3 \\
4 & 3 & 3 & 4 & 4
\end{array} \]\[\begin{array}{r|rrrr}
\tr^3 & 1 & 2 & 3 & 4 \\\hline
1 & 1 & 3 & 1 & 3 \\
2 & 4 & 2 & 4 & 2 \\
3 & 3 & 1 & 3 & 1 \\
4 & 2 & 4 & 2 & 4
\end{array} \ \ \begin{array}{r|rrrr}
\tr^4 & 1 & 2 & 3 & 4 \\\hline
1 & 1 & 4 & 4 & 1 \\
2 & 3 & 2 & 2 & 3 \\
3 & 2 & 3 & 3 & 2 \\
4 & 4 & 1 & 1 & 4
\end{array} \ \ \begin{array}{r|rrrr}
\tr^5 & 1 & 2 & 3 & 4 \\\hline
1 & 1 & 3 & 4 & 2 \\
2 & 4 & 2 & 1 & 3 \\
3 & 2 & 4 & 3 & 1 \\
4 & 3 & 1 & 2 & 4
\end{array} \ \ \begin{array}{r|rrrr}
\tr^6 & 1 & 2 & 3 & 4 \\\hline
1 & 1 & 4 & 2 & 3 \\
2 & 3 & 2 & 4 & 1 \\
3 & 4 & 1 & 3 & 2 \\
4 & 2 & 3 & 1 & 4
\end{array}
\]
and computed the $G$-family polynomial 
enhancement for each of the genus 2 handlebody-knots in the table in \cite{IKMS}. 
The results are collected in the tables.

\[\begin{array}{r|l}
K & \Phi_{(G,X)}(K) \\ \hline
4_1 & \{216\times [t^4s^4, t^2s^2, t^2s^2, t^2s^2, ts, ts], 96\times [4t^4s^4, 4t^2s^2, 4t^2s^2, 4t^2s^2, 4ts, 4ts]\} \\
5_1 & \{24\times [t^4s^4, t^2s^2, t^2s^2, t^2s^2, ts, ts]\} \\
5_2 & \{216\times [t^4s^4, t^2s^2, t^2s^2, t^2s^2, ts, ts], 96\times [4t^4s^4, 4t^2s^2, 4t^2s^2, 4t^2s^2, 4ts, 4ts]\}\\ 
5_3 & \{48\times [t^4s^4, t^2s^2, t^2s^2, t^2s^2, ts, ts]\}\\ 
5_4 & \{144\times [t^4s^4, t^2s^2, t^2s^2, t^2s^2, ts, ts]\}\\ 
6_1 & \{144\times [t^4s^4, t^2s^2, t^2s^2, t^2s^2, ts, ts], 
48\times [4t^4s^4, 4t^2s^2, 4t^2s^2, 4t^2s^2, 4ts, 4ts]\}, \\
6_2 & \{120\times[t^4s^4, t^2s^2, t^2s^2, t^2s^2, ts, ts], 
 24\times [4t^4s^4, 4t^2s^2, 4t^2s^2, 4t^2s^2, 4ts, 4ts]\}\\ 
6_3 & \{24\times[t^4s^4, t^2s^2, t^2s^2, t^2s^2, ts, ts]\} \\ 
6_4 & \{72\times [t^4s^4, t^2s^2, t^2s^2, t^2s^2, ts, ts]\} \\
6_5 & \{24\times[t^4s^4, t^2s^2, t^2s^2, t^2s^2, ts, ts]\} \\ 
6_6 & \{24\times[t^4s^4, t^2s^2, t^2s^2, t^2s^2, ts, ts]\} \\ 
6_7 & \{24\times[t^4s^4, t^2s^2, t^2s^2, t^2s^2, ts, ts]\} \\ 
6_8 & \{24\times[t^4s^4, t^2s^2, t^2s^2, t^2s^2, ts, ts]\} \\ 
6_9 & \{72\times [t^4s^4, t^2s^2, t^2s^2, t^2s^2, ts, ts]\}\\
6_{10} & \{216\times [t^4s^4, t^2s^2, t^2s^2, t^2s^2, ts, ts], 72\times [4t^4s^4, 4t^2s^2, 4t^2s^2, 4t^2s^2, 4ts, 4ts]\}\\ 
6_{11} & \{48\times [t^4s^4, t^2s^2, t^2s^2, t^2s^2, ts, ts]\}\\
6_{12} & \{24\times [t^4s^4, t^2s^2, t^2s^2, t^2s^2, ts, ts]\} \\
6_{13} & \{144\times [t^4s^4, t^2s^2, t^2s^2, t^2s^2, ts, ts]\} \\
6_{14} &\{144\times [t^4s^4, t^2s^2, t^2s^2, t^2s^2, ts, ts], 72\times [4t^4s^4, 4t^2s^2, 4t^2s^2, 4t^2s^2, 4ts, 4ts]\}.
\end{array}\]

In particular, $\Phi_{(G,X)}$ distinguishes  the handlebody-knot $6_2$
from $6_{13}$ and $5_4$  despite the counting invariants being equal.
We further note that with this $S_3$-family of quandles, the associated 
quandle version of our invariant further distinguishes handlebody-knots
$4_1$ and $5_2$ despite both having the same number of colorings.
\[\begin{array}{r|l}
K & \Phi_{(G,X)}^A(K) \\ \hline
4_1 & \{4\times t^{24}s^{12}, 36\times(t^{24}s^{12} + t^4s^6), 24\times(t^{24}s^{12} + t^3s^6), 24\times(t^{24}s^{12} + 3t^4s^6), \\ & \ \ 48\times(t^{24}s^{12} + 3t^4s^6 + 2t^3s^6), 24\times(3t^{24}s^{12} + 3t^4s^6), 48\times(4t^{24}s^{12} + 12t^4s^6 + 8t^3s^6),\\ & \ \  72\times(3t^4s^6 + 2t^3s^6),  8\times(t^{24}s^{12} + 2t^3s^6), 24\times(4t^{24}s^{12} + 4t^3s^6)\} \\
5_2 & \{4\times t^{24}s^{12}, 36\times(t^{24}s^{12} + t^4s^6), 24\times(t^{24}s^{12} + 3t^4s^6), 16\times(t^{24}s^{12} + t^3s^6), 24\times(3t^{24}s^{12} + 3t^4s^6),\\  & \ \  120\times(3t^4s^6 + 2t^3s^6), 48\times(12t^4s^6 + 8t^3s^6), 8\times(t^{24}s^{12} + 2t^3s^6), 8\times 2t^3s^6, 24\times(8t^3s^6)\} \\
\end{array}\]

\end{example}

\section{\large\textbf{Questions}}\label{Q}

We conclude with some questions for future research.

\begin{itemize}
\item
$G$-families of quandles have a number of generalizations such as 
\textit{multiple conjugation quandles} \cite{I2}
and \textit{multiple conjugation biquandles} \cite{IIKKMO}. 
$G$-family polynomials should be likewise extendable to these structures.
\item What additional enhancements of the $G$-family counting invariant
are possible, either alone or in combination with $G$-family polynomials? 
\item What further refinements can be made to the $G$-family polynomial 
definition, perhaps allowing for subgroups for $G$?
\item In \cite{I} the notion of a $Q$-family of quandles for a quandle $Q$
is also introduced, suggesting an obvious notion of $Q$-family polynomials; 
these should give invariants analogous to the ones in this paper.
\end{itemize}

\bibliography{mb-sn}{}

\begin{thebibliography}{10}

\bibitem{CN}
T.~Carrell and S.~Nelson.
\newblock On rack polynomials.
\newblock {\em J. Algebra Appl.}, 10(6):1221--1232, 2011.

\bibitem{EN}
M.~Elhamdadi and S.~Nelson.
\newblock {\em Quandles---an introduction to the algebra of knots}, volume~74
  of {\em Student Mathematical Library}.
\newblock American Mathematical Society, Providence, RI, 2015.

\bibitem{I2}
A.~Ishii.
\newblock A multiple conjugation quandle and handlebody-knots.
\newblock {\em Topology Appl.}, 196(part B):492--500, 2015.

\bibitem{I}
A.~Ishii, M.~Iwakiri, Y.~Jang, and K.~Oshiro.
\newblock A {$G$}-family of quandles and handlebody-knots.
\newblock {\em Illinois J. Math.}, 57(3):817--838, 2013.

\bibitem{IIKKMO}
A.~Ishii, M.~Iwakiri, S.~Kamada, J.~Kim, S.~Matsuzaki, and K.~Oshiro.
\newblock A multiple conjugation biquandle and handlebody-links.
\newblock {\em Hiroshima Math. J.}, 48(1):89--117, 2018.

\bibitem{IKMS}
A.~Ishii, K.~Kishimoto, H.~Moriuchi, and M.~Suzuki.
\newblock A table of genus two handlebody-knots up to six crossings.
\newblock {\em J. Knot Theory Ramifications}, 21(4):1250035, 9, 2012.

\bibitem{J}
D.~Joyce.
\newblock A classifying invariant of knots, the knot quandle.
\newblock {\em J. Pure Appl. Algebra}, 23(1):37--65, 1982.

\bibitem{M}
S.~V. Matveev.
\newblock Distributive groupoids in knot theory.
\newblock {\em Mat. Sb. (N.S.)}, 119(161)(1):78--88, 160, 1982.

\bibitem{N}
S.~Nelson.
\newblock A polynomial invariant of finite quandles.
\newblock {\em J. Algebra Appl.}, 7(2):263--273, 2008.

\bibitem{N2}
S.~Nelson.
\newblock Generalized quandle polynomials.
\newblock {\em Canad. Math. Bull.}, 54(1):147--158, 2011.

\end{thebibliography}
\bibliographystyle{abbrv}

\bigskip

\noindent
\textsc{Department of Mathematical Sciences \\
Claremont McKenna College \\
850 Columbia Ave. \\
Claremont, CA 91711} 

\

\end{document}